\documentclass[12pt]{article}
\pdfoutput=1

\usepackage{amsmath,amssymb,calc}
\usepackage{amsthm}
\usepackage{thmtools}
\usepackage{mathrsfs}
\usepackage{bm}
\usepackage{bbm}
\usepackage{mathtools}
\usepackage{enumerate}
\usepackage[authoryear]{natbib}
\usepackage{caption,subcaption}
\usepackage{pgfplots}
\usepackage{chngcntr}
\usepackage{appendix}
\usepackage[colorlinks,citecolor=blue,urlcolor=blue]{hyperref}
\usepackage{cleveref}
\usepackage{crossreftools}
\usepackage{titlesec}
\usepackage{hyphenat}
\usepackage{booktabs}

\usepackage[T1]{fontenc}
\usepackage{textcomp}
\usepackage{lmodern}

\usepackage{tikz}
\usetikzlibrary{arrows,arrows.meta}
\usetikzlibrary{tikzmark}
\usepackage{graphicx}
\usepackage{centernot}
\usetikzlibrary{nfold}

\usetikzlibrary{decorations.markings}

\newlength\aetmplength

\usepackage{letltxmacro}

\LetLtxMacro\OriginalLongrightarrow\Longrightarrow
\LetLtxMacro\OriginalLongleftarrow\Longleftarrow

\usepackage{trimclip}
\DeclareRobustCommand\Longrightarrow{\NewRelbar\joinrel\Rightarrow}
\DeclareRobustCommand\Longleftarrow{\Leftarrow\joinrel\NewRelbar}

\makeatletter
\DeclareRobustCommand\NewRelbar{%
  \mathrel{%
    \mathpalette\@NewRelbar{}%
  }%
}
\newcommand*\@NewRelbar[2]{%
  \sbox0{$#1=$}%
  \sbox2{$#1\Rightarrow\m@th$}%
  \sbox4{$#1\Leftarrow\m@th$}%
  \clipbox{0pt 0pt \dimexpr(\wd2-.6\wd0) 0pt}{\copy2}%
  \kern-.2\wd0 %
  \clipbox{\dimexpr(\wd4-.6\wd0) 0pt 0pt 0pt}{\copy4}%
}
\makeatother

\usepackage{authblk}
\author[1]{Tetsuya Kaji}
\affil[1]{University of Chicago}

\title{The Hellinger Bounds on the Kullback--Leibler Divergence and the Bernstein Norm%
\thanks{This work is supported by the Richard N.\ Rosett Faculty Fellowship at the University of Chicago Booth School of Business.}}
\date{\today}

\allowdisplaybreaks[3]

\theoremstyle{plain}
\newtheorem{thm}{Theorem}
\newtheorem{prop}[thm]{Proposition}

\theoremstyle{definition}
\newtheorem*{defn}{Definition}

\theoremstyle{remark}
\newtheorem*{rem}{Remark}

\renewcommand\thmcontinues[1]{continued}

\newcommand{\exasymbol}{$\square$}

\declaretheoremstyle[
spaceabove=6pt, spacebelow=6pt,
headfont=\normalfont\bfseries,
notefont=\mdseries, notebraces={(}{)},
bodyfont=\normalfont,
postheadspace=1em,
qed=\exasymbol
]{exa}
\declaretheorem[style=exa,name=Example]{exa}

\counterwithin{subsection}{section}
\counterwithin{subsubsection}{subsection}

\crefname{thm}{Theorem}{Theorems}
\crefname{prop}{Proposition}{Propositions}
\crefname{exa}{Example}{Examples}
\crefname{section}{Section}{Sections}
\crefname{subsection}{Section}{Sections}
\crefname{subsubsection}{Section}{Sections}

\def\midd{\mathrel{\Vert}}
\DeclareMathOperator{\Var}{Var}

\DeclareMathOperator*{\conv}{\mathchoice{%
	\,\longrightarrow\,}{%
	\rightarrow}{%
	\rightarrow}{%
	\rightarrow}%
}

\usepackage[margin=1.25in]{geometry}
\usepackage[onehalfspacing,nodisplayskipstretch]{setspace}

\usepackage{accents}
\newcommand{\ubar}[1]{\underaccent{\bar}{#1}}

\begin{document}

\maketitle

\begin{abstract}
The Kullback--Leibler divergence, the Kullback--Leibler variation, and the Bernstein ``norm'' are used to quantify discrepancies among probability distributions in likelihood models such as nonparametric maximum likelihood and nonparametric Bayes.
They are closely related to the Hellinger distance, which is often easier to work with.
Consequently, it is of interest to characterize conditions under which the Hellinger distance serves as an upper bound for these measures.
This article characterizes a necessary and sufficient condition for each of the discrepancy measures to be bounded by the Hellinger distance.
It accommodates unbounded likelihood ratios and generalizes all previously known results.
We then apply it to relax the regularity condition for the sieve maximum likelihood estimator.
\end{abstract}

\section{Introduction}

Controlling the size of a function class is a central step in nonparametric statistics.
The Kullback--Leibler divergence, the Kullback--Leibler variation, and the Bernstein ``norm'' are standard measures of discrepancy in minimum contrast estimation \citep[p.~433]{vw2023} and nonparametric Bayesian analysis \citep{gv2017}.
These quantities are closely related to the Hellinger distance, which endows the space of probability distributions with a convenient Hilbert space structure.
A well\hyp{}known fact is that the Kullback--Leibler divergence bounds the Hellinger distance from above---an inequality often used, for example, to establish identification in maximum likelihood estimation \citep[Lemma 5.35]{v1998}.
The reverse inequality generally fails to hold.
For complexity control of function classes, however, the converse is typically the more wanted direction.
This asymmetry explains why general posterior contraction theorems involve two distinct types of neighborhoods \citep[p.~199]{gv2017}.

Hence, it is natural to seek sufficient conditions under which the Hellinger distance serves as an upper bound for the three discrepancy measures.
\citet[Lemma 4.4]{b1983} and \citet[(7.6)]{bm1998} showed that the Kullback--Leibler divergence is bounded by the Hellinger distance when the likelihood ratio is uniformly bounded.
\Citet[p.~327]{vw1996} and \citet[Lemma 8.3]{ggv2000} showed that the Bernstein ``norm'' of the log-likelihood ratio is bounded under the same condition.
\citet[Lemma B.3 (B.4)]{gv2017} provided the corresponding bound on the Kullback--Leibler variation.
The uniform boundedness condition, however, is arguably restrictive---particularly in models involving unbounded random variables.
Even the canonical normal location model violates this requirement.

Some attempts have been made to relax this condition.
\citet[Theorem 5]{ws1995} showed that the Kullback--Leibler divergence and variation are bounded by the Hellinger distance when a local moment of the likelihood ratio is bounded by the Hellinger distance.
This result is general enough to accommodate many models with unbounded likelihood ratios but is somewhat underused in later literature; e.g., \citet[Lemma 8.6]{ggv2000} reformulated it in a way that the multiples on the Hellinger distance diverge as the Hellinger distance tends to zero.
When the multiples diverge, the subsequent statistical application may lead to a compromised rate of convergence or contraction.

For the Bernstein ``norm'', \citet[Lemma 8.7]{ggv2000} proved an analogous bound under a similar condition for which the multiple diverges.
\citet[Lemma S.4]{kmp2023} showed that the Bernstein ``norm'' of half the log-likelihood ratio is bounded by the Hellinger distance under the condition that the likelihood ratio has a finite local moment conditional on the likelihood ratio exceeding a threshold.
This was the first to accommodate unbounded likelihood ratios for a sharp Hellinger bound on the Bernstein ``norm''.
\citet[Lemma 2.1]{kr2023} showed that the Kullback--Leibler divergence and variation are bounded under the same condition.

In this article, we identify the necessary and sufficient condition for each of the discrepancy measures to be bounded by the Hellinger distance.
We verify that the aforementioned sufficient conditions imply our necessary and sufficient conditions.
We also study the relationship between existing conditions in the literature.
Finally, we apply the results to obtain the rate of convergence of a nonparametric sieve maximum likelihood estimator under a relaxed regularity requirement, which permits the likelihood ratio to diverge faster than previously possible.

The paper is organized as follows.
\cref{sec:setup} introduces notation and definitions.
\cref{sec:bounds} develops the necessary and sufficient condition for each of the discrepancy measures to be bounded by the Hellinger distance.
\cref{sec:comparison} compares the existing conditions and the new ones.
\cref{sec:mle} applies the results to establish a new result of the convergence rate for a sieve maximum likelihood estimator.
\cref{sec:conc} concludes.

\section{Definition} \label{sec:setup}

We work with probability measures defined on a measurable space $(\mathcal{X},\mathcal{A})$.
A probability measure is denoted with a capital letter (e.g., $P$), and the corresponding density (Radon--Nikodym derivative with respect to some $\sigma$\hyp{}finite dominating measure) with a lowercase letter (e.g., $p$).
All integrals considered in this paper do not depend on the particular choice of the dominating measure; hence it is made implicit.
For example, $\int(p-q)$ means $\int_{\mathcal{X}}(p(x)-q(x))d\mu(x)$, where $\mu$ is any $\sigma$\hyp{}finite measure dominating both $P$ and $Q$.
Expectations are often written using operator notation:
\[
	P f\coloneqq\mathbb{E}_{X\sim P}[f(X)]=\int_{\mathcal{X}}f(x)dP(x).
\]

We also write $P(A)$ to denote the probability of an event $A\in\mathcal{A}$ under $P$.
Hence, for a measurable function $f$, $P(f)$ means $\mathbb{E}_{X\sim P}[f(X)]$ while $P(f\geq 0)$ refers to $\mathbb{E}_{X\sim P}[\mathbbm{1}\{f(X)\geq 0\}]$.
The notation $P(f\mid f\geq 0)$ denotes the conditional expectation $\mathbb{E}_{X\sim P}[f(X)\mid f(X)\geq 0]$, which we define to be $0$ if $P(f\geq 0)=0$.

The discrepancy measures of interest are defined as follows.

\begin{defn}[Hellinger distance]
The {\em Hellinger distance} between two probability measures $P$ and $Q$ is%
\footnote{Some authors include $1/2$ inside the integral \citep[e.g.,][]{b1983,bm1998}.}
\[
	h(p,q)\coloneqq\Bigl[\int(\sqrt{p}-\sqrt{q})^2\Bigr]^{1/2}
	=\Bigl[\int_{\mathcal{X}}\bigl(\sqrt{p(x)}-\sqrt{q(x)}\bigr)^2 d\mu(x)\Bigr]^{1/2}.
\]
\end{defn}

\begin{defn}[Kullback--Leibler divergence and variation]
The {\em Kullback--Leibler divergence} of $Q$ from $P$ is
\[
	K(p\midd q)\coloneqq P\log\frac{p}{q}=\int_{\mathcal{X}}\biggl(\log\frac{p(x)}{q(x)}\biggr)dP(x).
\]
For $k>1$, the {\em $k$th order Kullback--Leibler variation} of $Q$ from $P$ is
\[
	V_k(p\midd q)\coloneqq P\biggl|\log\frac{p}{q}\biggr|^k=\int_{\mathcal{X}}\biggl|\log\frac{p(x)}{q(x)}\biggr|^k dP(x),
\]
and the {\em $k$th order centered Kullback--Leibler variation} of $Q$ from $P$ is
\[
	V_{k,0}(p\midd q)\coloneqq P\biggl|\log\frac{p}{q}-K(p\midd q)\biggr|^k=\int_{\mathcal{X}}\biggl|\log\frac{p(x)}{q(x)}-P\log\frac{p}{q}\biggr|^k dP(x).
\]
\end{defn}

If $P(A)=0$ and $Q(A)>0$, the event $A$ is ignored in these integrals.
Meanwhile, if $P(A)>0$ and $Q(A)=0$, we may understand the integrals to be infinity.

The following Bernstein ``norm'' was introduced in \citet[Notes 3.4]{vw1996} to apply Bernstein's inequality for a maximal inequality to obtain the rate of convergence of an $M$-estimator.

\begin{defn}[Bernstein ``norm'']
The {\em Bernstein ``norm''} of a measurable function $f$ is %
\[
	\lVert f\rVert_{P,B}\coloneqq\sqrt{2P(e^{\lvert f\rvert}-1-\lvert f\rvert)}=\Bigl[2\int_{\mathcal{X}}\bigl(e^{\lvert f(x)\rvert}-1-\lvert f(x)\rvert\bigr)dP(x)\Bigr]^{1/2}.
\]
\end{defn}

This is not a true norm, as it neither is homogeneous nor satisfies the triangle inequality \citep[p.~338]{vw2023}, but satisfies the so-called Riesz property: $\lvert f\rvert\leq\lvert g\rvert$ implies $\lVert f\rVert\leq\lVert g\rVert$ \citep[p.~132]{vw2023}.
The idea is to control exponential deviation of a function.
Since $x^2\leq 2(e^{\lvert x\rvert}-1-\lvert x\rvert)$ for every $x\in\mathbb{R}$, we have $\lVert f\rVert_2\leq\lVert f\rVert_{P,B}$.
Moreover, as the exponential function grows faster than any polynomial, this dominates all $L^p$ norms for $p\geq 2$ up to a constant depending only on $p$.
The Bernstein ``norm'' naturally relates to Bernstein's inequality \citep[Lemma 2.2.10]{vw2023} and is particularly useful for minimum contrast estimation \citep[p.~433]{vw2023}.

We also make use of a ``norm'' that is equivalent to the Bernstein ``norm'' but is much more convenient when dealing with log-likelihood ratios.
Since for every $x\in\mathbb{R}$,
\begin{equation} \label{eq:norm}
	x^2\leq e^x+e^{-x}-2\leq 2(e^{\lvert x\rvert}-1-\lvert x\rvert)\leq 2(e^x+e^{-x}-2),
\end{equation}
we have
\[
	\lVert f\rVert_2\leq\sqrt{P(e^f+e^{-f}-2)}\leq\lVert f\rVert_{P,B}\leq\sqrt{2 P(e^f+e^{-f}-2)}.
\]
Therefore, the ``norm'' $\sqrt{P(e^f+e^{-f}-2)}$ is equivalent to the Bernstein ``norm''.
However, this goes well with log-likelihood ratios as it comes with various identities
\begin{equation} \label{eq:convenient}
	e^f+e^{-f}-2=
	\begin{cases}(e^f-1)(1-e^{-f}),\\
	(e^{f/2}-1)^2(1+e^{-f/2})^2,\\
	(e^{-f/2}-1)^2(1+e^{f/2})^2,\end{cases}
\end{equation}
which are particularly useful when $f$ is a log-likelihood ratio.
Finally, we occasionally use notation $x\vee y\coloneqq\max\{x,y\}$ and $x\wedge y\coloneqq\min\{x,y\}$.

\section{The Hellinger Bounds} \label{sec:bounds}

\subsection{Bernstein ``Norm''}

\Citet[Chapter 3.4]{vw1996} introduced the Bernstein ``norm'' to derive a type of maximal inequality---a bound on the supremum of the empirical process over a class of functions.
Apparently, the empirical process evaluated at a function is homogeneous in the sense that bounding $\lVert\mathbb{G}_n\rVert_{\mathcal{F}}=\sup_{f\in\mathcal{F}}\lvert\sqrt{n}(\mathbb{P}_n-P)f\rvert$ is equivalent to bounding $2\lVert\mathbb{G}_n\rVert_{2^{-1}\mathcal{F}}=2\sup_{f\in\mathcal{F}}\lvert\sqrt{n}(\mathbb{P}_n-P)(f/2)\rvert$.
However, the Bernstein ``norm'' is not homogeneous in that we always have $\lVert f\rVert_{P_0,B}\geq 2\lVert 2^{-1}f\rVert_{P_0,B}$ where the RHS can even be finite when the LHS is not.
Therefore, bounding the halved Bernstein ``norm'' is always at least as easy as bounding the original Bernstein ``norm''.
This fact was mentioned in \citet[p.~325]{vw1996} and was used by \citet{kmp2023} and \citet{kr2023} to prove the maximal inequality for a nonparametric classifier.
In this spirit, we present the necessary and sufficient condition for bounding the arbitrary {\em fractional} Bernstein ``norm''---$\lVert\delta f\rVert_{P_0,B}$ for any $\delta\in(0,1]$---when $f$ is a log-likelihood ratio.
In practice, it suffices to meet the condition for just one $\delta\in(0,1]$ to secure a maximal inequality.

Let $\delta\in(0,1]$.
The necessary and sufficient condition for the fractional Bernstein ``norm'' of the log-likelihood ratio to be bounded by the Hellinger distance,
\begin{equation} \label{eq:BN} \tag{BN}
	\biggl\|\delta\log\frac{p_0}{p}\biggr\|_{P_0,B}^2\lesssim h(p_0,p)^2,
\end{equation}
is found to be
\begin{equation} \label{asm:BN} \tag{NC}
	P_0\biggl(\biggl[\frac{p_0}{p}\biggr]^\delta\mathbbm{1}\biggl\{\frac{p_0}{p}>4\biggr\}\biggr)\lesssim h(p_0,p)^2.
\end{equation}
Precisely, the following theorem shows that when $P_0([\frac{p_0}{p}]^\delta\mathbbm{1}\{\frac{p_0}{p}>4\})\leq M h(p_0,p)^2$ holds, we have $\lVert\delta\log\frac{p_0}{p}\rVert_{P_0,B}^2\leq(18\delta+2 M)h(p_0,p)^2$, along with the other direction with a different multiple.
It also shows that the Kullback--Leibler divergence and second- or higher-order Kullback--Leibler variation are bounded by the Hellinger distance.
These additional bounds are sharp in $h$ but not in the multiples.
They are enough when we assume (\ref{asm:BN}), but if we wish to allow the multiples to diverge or if we do not need the Bernstein ``norm'', we may instead want to use \cref{lem:KL3} below.

\begin{thm}[Bernstein ``norm''; $\text{(\ref{asm:BN})}\Leftrightarrow\text{(\ref{eq:BN})}$] \label{thm:BN}
For arbitrary probability measures $P_0$ and $P$ and $\delta\in(0,1]$,
\[
	\biggl(1-\frac{1}{4^\delta}\biggr)^2 P_0\biggl(\frac{p_0^\delta}{p^\delta}\mathbbm{1}\biggl\{\frac{p_0}{p}>4\biggr\}\biggr)
	\leq \biggl\|\delta\log\frac{p_0}{p}\biggr\|_{P_0,B}^2
	\leq 18\delta h(p_0,p)^2+2 P_0\biggl(\frac{p_0^\delta}{p^\delta}\mathbbm{1}\biggl\{\frac{p_0}{p}>4\biggr\}\biggr).
\]
Moreover, we have
\begin{enumerate}[(i)]
	\item \label{thm:BN:KD}
		$h(p_0,p)^2\leq K(p_0\midd p)\leq 3 h(p_0,p)^2+\delta^{-1}P_0([\frac{p_0}{p}]^\delta\mathbbm{1}\{\frac{p_0}{p}>4\})$,
	\item \label{thm:BN:KH}
		$2^{-k}V_{k,0}(p_0\midd p)\leq V_k(p_0\midd p)\leq 2^{-1}\Gamma(k+1)\delta^{-k}\lVert\delta\log\tfrac{p_0}{p}\rVert_{P_0,B}^2$
		for every real $k\geq 2$.
\end{enumerate}
\end{thm}

\begin{rem}
For $k=2$, $\Var(X)\leq\mathbb{E}[X^2]$ gives a better bound $V_{2,0}(p_0\midd p)\leq V_2(p_0\midd p)$.
\end{rem}

\begin{proof}
$\text{(\ref{eq:BN})}\Rightarrow\text{(\ref{asm:BN})}$.
We first show that $(\sqrt{x^\delta}-1)^2\leq\delta(\sqrt{x}-1)^2$ for $x\geq\frac{1}{4}$ and $0<\delta\leq 1$.
Consider $x\geq 1$, so both $(\sqrt{x^\delta}-1)^2$ and $(\sqrt{x}-1)^2$ are increasing and they coincide at $x=1$.
The derivatives are, respectively,
\[
	\delta\tfrac{1}{x^{1-\delta}}-\tfrac{\delta}{\sqrt{x}}\tfrac{1}{x^{\frac{1-\delta}{2}}}, \qquad
	\delta-\tfrac{\delta}{\sqrt{x}}.
\]
Since $1\leq x^{\frac{1-\delta}{2}}\leq x^{1-\delta}$, we see that
\[
	\delta-\tfrac{\delta}{\sqrt{x}}\geq\bigl(\delta-\tfrac{\delta}{\sqrt{x}}\bigr)\tfrac{1}{x^{1-\delta}}\geq\delta\tfrac{1}{x^{1-\delta}}-\tfrac{\delta}{\sqrt{x}}\tfrac{1}{x^{\frac{1-\delta}{2}}},
\]
so $(\sqrt{x}-1)^2$ grows faster than $(\sqrt{x^\delta}-1)^2$.
Thus, we have $(\sqrt{x^\delta}-1)^2\leq\delta(\sqrt{x}-1)^2$.

Next, consider $\frac{1}{4}\leq x<1$ so both are decreasing.
As long as $(\sqrt{x}-1)^2$ decreases faster, we have the desired inequality.
That is, we want
\[
	\Bigl(\delta\tfrac{1}{x^{1-\delta}}-\tfrac{\delta}{\sqrt{x}}\tfrac{1}{x^{\frac{1-\delta}{2}}}\Bigr)-\bigl(\delta-\tfrac{\delta}{\sqrt{x}}\bigr)
	=\Bigl[\delta\Bigl(\tfrac{1}{x^{\frac{1-\delta}{2}}}+1\Bigr)-\tfrac{\delta}{\sqrt{x}}\Bigr]\Bigl(\tfrac{1}{x^{\frac{1-\delta}{2}}}-1\Bigr)\geq 0.
\]
This is equivalent to
\[
	\Bigl(\tfrac{1}{x^{\frac{1-\delta}{2}}}+1\Bigr)-\tfrac{1}{\sqrt{x}}\geq 0
	\iff
	\delta\geq\tfrac{\log(1-\sqrt{x})}{\log\sqrt{x}}.
\]
The RHS is increasing and equal to one at $x=\frac{1}{4}$.
Since $\delta\leq 1$, this holds on $\frac{1}{4}\leq x<1$ and so does our desired inequality. Thus, $(\sqrt{x^\delta}-1)^2\leq\delta(\sqrt{x}-1)^2$ for $x\geq\frac{1}{4}$.

Now we bound the fractional Bernstein ``norm''.
Since $x^\delta+\frac{1}{x^\delta}-2<x^\delta$ for $x>4$, %
\begin{align*}
	\bigl\|\delta\log\tfrac{p_0}{p}\bigr\|_{P_0,B}^2&\leq 2 P_0\bigl(\bigl[\tfrac{p_0}{p}\bigr]^\delta+\bigl[\tfrac{p}{p_0}\bigr]^\delta-2\bigr) \tag{by (\ref{eq:norm})} \\
	&\leq 2 P_0\bigl(\bigl[\tfrac{p}{p_0}\bigr]^{\delta/2}-1\bigr)^2\bigl(1+\bigl[\tfrac{p_0}{p}\bigr]^{\delta/2}\bigr)^2\mathbbm{1}\bigl\{\tfrac{p_0}{p}\leq 4\bigr\}
	+2 P_0\bigl[\tfrac{p_0}{p}\bigr]^\delta\mathbbm{1}\bigl\{\tfrac{p_0}{p}>4\bigr\} \tag{by (\ref{eq:convenient})} \\
	&\leq 2 P_0\delta\bigl(\sqrt{\tfrac{p}{p_0}}-1\bigr)^2 (1+2)^2+2 P_0\bigl[\tfrac{p_0}{p}\bigr]^\delta\mathbbm{1}\bigl\{\tfrac{p_0}{p}>4\bigr\} \tag{paragraphs above} \\
	&\leq 18\delta h(p_0,p)^2+2 P_0\bigl[\tfrac{p_0}{p}\bigr]^\delta\mathbbm{1}\bigl\{\tfrac{p_0}{p}>4\bigr\}.
\end{align*}
$\text{(\ref{eq:BN})}\Leftarrow\text{(\ref{asm:BN})}$.
Since $(1-\frac{1}{4^\delta})^2 x^\delta<(1-\frac{1}{x^\delta})^2 x^\delta=x^\delta+\frac{1}{x^\delta}-2$ for $x>4$, by (\ref{eq:norm}),
\[
	\bigl(1-\tfrac{1}{4^\delta}\bigr)^2 P_0\bigl[\tfrac{p_0}{p}\bigr]^\delta\mathbbm{1}\bigl\{\tfrac{p_0}{p}>4\bigr\}
	\leq P_0\bigl(\bigl[\tfrac{p_0}{p}\bigr]^\delta+\bigl[\tfrac{p}{p_0}\bigr]^\delta-2\bigr)\mathbbm{1}\bigl\{\tfrac{p_0}{p}>4\bigr\}
	\leq \bigl\|\delta\log\tfrac{p_0}{p}\bigr\|_{P_0,B}^2.
\]

(\ref{thm:BN:KD})
The lower bound follows from $(\sqrt{x}-1)^2\leq x-1-\log x$ as
\[
	h(p_0,p)^2=P_0\bigl(\sqrt{\tfrac{p}{p_0}}-1\bigr)^2+P(p_0=0)
	\leq P_0\bigl(\tfrac{p}{p_0}-1-\log\tfrac{p}{p_0}\bigr)+P(p_0=0)=K(p_0\midd p).
\]
For the upper bound, since $x-1-\log x\leq 3(\sqrt{x}-1)^2$ for $x\geq\frac{1}{4}$, we can write
\begin{align*}
	K(p_0\midd p)&=P_0\bigl(\tfrac{p}{p_0}-1-\log\tfrac{p}{p_0}\bigr)+P(p_0=0)\\
	&\leq 3 P_0\bigl(\sqrt{\tfrac{p}{p_0}}-1\bigr)^2\mathbbm{1}\bigl\{\tfrac{p}{p_0}\geq\tfrac{1}{4}\bigr\}+P(p_0=0)+P_0\bigl(\tfrac{p}{p_0}-1-\log\tfrac{p}{p_0}\bigr)\mathbbm{1}\bigl\{\tfrac{p_0}{p}>4\bigr\}\\
	&\leq 3 h(p_0,p)^2+P_0\log\tfrac{p_0}{p}\mathbbm{1}\bigl\{\tfrac{p_0}{p}>4\bigr\}.
\end{align*}
Finally,
\(
	P_0\log\tfrac{p_0}{p}\mathbbm{1}\{\tfrac{p_0}{p}>4\}
	=\delta^{-1}P_0\delta\log\tfrac{p_0}{p}\mathbbm{1}\{\tfrac{p_0}{p}>4\}
	\leq\delta^{-1}P_0[\tfrac{p_0}{p}]^\delta\mathbbm{1}\{\tfrac{p_0}{p}>4\}
\).

(\ref{thm:BN:KH})
For the first inequality, observe that for $k\geq 1$,
\begin{align*}
	V_{k,0}(p_0\midd p)^{1/k}&=
	\bigl(P_0\bigl|\log\tfrac{p_0}{p}-P_0\log\tfrac{p_0}{p}\bigr|^k\bigr)^{1/k}\\
	&\leq\bigl(P_0\bigl|\log\tfrac{p_0}{p}\bigr|^k\bigr)^{1/k}+\bigl|P_0\log\tfrac{p_0}{p}\bigr| \tag{triangle inequality} \\
	&\leq\bigl(P_0\bigl|\log\tfrac{p_0}{p}\bigr|^k\bigr)^{1/k}+\bigl(P_0\bigl|\log\tfrac{p_0}{p}\bigr|^k\bigr)^{1/k} \tag{Jensen's inequality} \\
	&=2 V_k(p_0\midd p)^{1/k}.
\end{align*}
For the second inequality, it suffices to show that $x^k/\Gamma(k+1)\leq e^x-1-x$ for $k\geq 2$ and $x\geq 0$, since then by letting $x=\lvert\delta\log\frac{p_0}{p}\rvert$ we obtain
\[
	P_0\bigl|\log\tfrac{p_0}{p}\bigr|^k=\delta^{-k} P_0\bigl|\delta\log\tfrac{p_0}{p}\bigr|^k\leq 2^{-1}\Gamma(k+1)\delta^{-k}\bigl\|\delta\log\tfrac{p_0}{p}\bigr\|_{P_0,B}^2.
\]
By the definition of the gamma function, for arbitrary $t\geq 0$,
\[
	\Gamma(k-1)=\int_0^\infty z^{k-2}e^{-z}dz
	\geq\int_t^\infty z^{k-2}e^{-z}dz
	\geq t^{k-2}\int_t^\infty e^{-z}dz=t^{k-2}e^{-t}.
\]
Thus, we deduce that
\[
	\tfrac{x^k}{\Gamma(k+1)}=\int_0^x\int_0^y\tfrac{t^{k-2}}{\Gamma(k-1)}dt dy
	\leq\int_0^x\int_0^y e^t dt dy=e^x-1-x.
\]
This completes the proof.
\end{proof}

In condition (\ref{asm:BN}) and throughout the paper, the threshold of four appears frequently.
This choice is somewhat arbitrary---it can be any fixed number strictly above $1$.
We choose a square number as we deal with many square roots.
Note that (\ref{asm:BN}) is {\em not} equivalent to the one without the cutoff,
\(
	P_0([\frac{p_0}{p}]^\delta)\lesssim h(p_0,p)^2
\);
since $\frac{p_0}{p}$ approaches one as $h(p_0,p)^2\to 0$, the LHS never vanishes.
A sufficient but not necessary condition for (\ref{asm:BN}) is $P_0(\lvert\sqrt{\tfrac{p_0}{p}}-1\rvert^{2\delta})\lesssim h(p_0,p)^2$.%
\footnote{It is necessary when $\delta=1$. Then, it means that the ``reverse'' Hellinger distance is of the same order as the ``forward'' Hellinger distance.}

\subsection{Kullback--Leibler Divergence and Variation}

The next question is when the Kullback--Leibler divergence and variation are bounded by the Hellinger distance, that is, for $k\geq 2$,
\begin{align}
	K(p_0\midd p)&\lesssim h(p_0,p)^2, \label{eq:KLD} \tag{KD} \\
	V_k(p_0\midd p)&\lesssim h(p_0,p)^2. \label{eq:KLV} \tag{KV}
\end{align}
We establish that the necessary and sufficient condition for each, respectively, is
\begin{align}
	P_0\biggl(\biggl[\log\frac{p_0}{p}\biggr]\mathbbm{1}\biggl\{\frac{p_0}{p}>4\biggr\}\biggr)&\lesssim h(p_0,p)^2, \label{asm:KLD} \tag{L1} \\
	P_0\biggl(\biggl[\log\frac{p_0}{p}\biggr]^k\mathbbm{1}\biggl\{\frac{p_0}{p}>4\biggr\}\biggr)&\lesssim h(p_0,p)^2. \label{asm:KLV} \tag{Lk}
\end{align}
Since the integrand of the divergence has positive and negative parts that cancel each other, the necessity of (\ref{asm:KLD}) is not trivial.
This implies that the second set of bounds given in \citet[Lemma B.2]{gv2017} cannot be improved.

Also note that the higher the $k$, the stronger the condition.
The third point of the next theorem tells that the multiples grow at different rates.
For example, if $P_0([\log\tfrac{p_0}{p}]^2\mathbbm{1}\{\tfrac{p_0}{p}>4\})\leq M h(p_0,p)^2$, we have $P_0([\log\tfrac{p_0}{p}]\mathbbm{1}\{\tfrac{p_0}{p}>4\})\leq 4\sqrt{M}h(p_0,p)^2$.

\begin{thm}[Kullback--Leibler divergence and variation] \label{lem:KL3}
For arbitrary probability measures $P_0$ and $P$, the following hold.
\begin{enumerate}[(i)]
	\item \label{lem:KL3:KD}
		$\text{(\ref{asm:KLD})}\Leftrightarrow\text{(\ref{eq:KLD})}$:
		\[
			\frac{1}{3}P_0\biggl(\biggl[\log\frac{p_0}{p}\biggr]\mathbbm{1}\biggl\{\frac{p_0}{p}>4\biggr\}\biggr)\leq K(p_0\midd p)\leq 3 h(p_0,p)^2+P_0\biggl(\biggl[\log\frac{p_0}{p}\biggr]\mathbbm{1}\biggl\{\frac{p_0}{p}>4\biggr\}\biggr).
		\]
	\item \label{lem:KL3:KV}
		$\text{(\ref{asm:KLV})}\Leftrightarrow\text{(\ref{eq:KLV})}$:
		For every real $k\geq 2$,
		\begin{multline*}
			P_0\biggl(\biggl[\log\frac{p_0}{p}\biggr]^k\mathbbm{1}\biggl\{\frac{p_0}{p}>4\biggr\}\biggr)\leq V_k(p_0\midd p)
			\leq 4\biggl([2(\log 4)^{k-2}]\vee\biggl[\frac{k}{e}\biggr]^k\biggr)h(p_0,p)^2\\
			+P_0\biggl(\biggl[\log\frac{p_0}{p}\biggr]^k\mathbbm{1}\biggl\{\frac{p_0}{p}>4\biggr\}\biggr).
		\end{multline*}
	\item \label{lem:KL3:BTW}
		$\text{(\ref{asm:KLV})}\Rightarrow\text{(\ref{asm:KLD})}$:
		For every real $k'\geq k>0$,
		\[
			P_0\biggl(\biggl[\log\frac{p_0}{p}\biggr]^k\mathbbm{1}\biggl\{\frac{p_0}{p}>4\biggr\}\biggr)
			\leq 4 h(p_0,p)^{2(1-\frac{k}{k'})}\biggl[P_0\biggl(\biggl[\log\frac{p_0}{p}\biggr]^{k'}\mathbbm{1}\biggl\{\frac{p_0}{p}>4\biggr\}\biggr)\biggr]^{\frac{k}{k'}}.
		\]
\end{enumerate}
\end{thm}

\begin{proof}
(\ref{lem:KL3:KD})
$\text{(\ref{asm:KLD})}\Rightarrow\text{(\ref{eq:KLD})}$.
It is shown in the proof of \cref{thm:BN} (\ref{thm:BN:KD}).

$\text{(\ref{asm:KLD})}\Leftarrow\text{(\ref{eq:KLD})}$.
Since $\log\frac{1}{x}<3(x-1-\log x)$ for $0<x<\frac{1}{4}$,
\[
	\tfrac{1}{3}P_0\log\tfrac{p_0}{p}\mathbbm{1}\bigl\{\tfrac{p_0}{p}>4\bigr\}
	\leq P_0\bigl(\tfrac{p}{p_0}-1-\log\tfrac{p}{p_0}\bigr)\leq P_0\log\tfrac{p_0}{p}.
\]

(\ref{lem:KL3:KV})
$\text{(\ref{asm:KLV})}\Rightarrow\text{(\ref{eq:KLV})}$.
Note that $(\log x)^2\leq 8(\sqrt{x}-1)^2$ for $x\geq\frac{1}{4}$.
Hence, for $\frac{1}{4}\leq x\leq 4$ and $k\geq 2$, we have
\(
	\lvert\log x\rvert^k\leq(\log 4)^{k-2}(\log x)^2\leq 8(\log 4)^{k-2}(\sqrt{x}-1)^2
\).
Now, we want $C_k$ such that
\[
	(\log x)^k\leq C_k(\sqrt{x}-1)^2
\]
for $x>4$.
This is equivalent to bounding
\(
	\sup_{x>4}\frac{(\log x)^k}{(\sqrt{x}-1)^2}
\).
Since
\(
	\frac{(\log x)^k}{(\sqrt{x}-1)^2}<4\frac{(\log x)^k}{x}
\)
for $x>4$, we have
\[
	C_k\leq 4\sup_{x>4}\frac{(\log x)^k}{x}.
\]
This is attained at $x=e^k$, so
\(
	C_k\leq 4(\tfrac{k}{e})^k
\).
Then, the bound follows with $x=\frac{p}{p_0}$.

$\text{(\ref{asm:KLV})}\Leftarrow\text{(\ref{eq:KLV})}$.
Trivially, $P_0(\log\frac{p_0}{p})^k\mathbbm{1}\{\frac{p_0}{p}>4\}\leq P_0\lvert\log\frac{p_0}{p}\rvert^k$.

(\ref{lem:KL3:BTW})
$\text{(\ref{asm:KLV})}\Rightarrow\text{(\ref{asm:KLD})}$.
It suffices to consider $k'>k$.
Since $\frac{1}{2}<1-\sqrt{x}\leq 1$ if $\frac{1}{x}>4$,
\begin{align*}
	P_0\bigl(\log\tfrac{p_0}{p}\bigr)^k\mathbbm{1}\bigl\{\tfrac{p_0}{p}>4\bigr\}
	&\leq 4 P_0\bigl(\log\tfrac{p_0}{p}\bigr)^k\mathbbm{1}\bigl\{\tfrac{p_0}{p}>4\bigr\}\bigl(1-\sqrt{\tfrac{p}{p_0}}\bigr)^{2\frac{k'-k}{k'}} \\
	&\leq 4\bigl[P_0\bigl(\log\tfrac{p_0}{p}\bigr)^{k'}\mathbbm{1}\bigl\{\tfrac{p_0}{p}>4\bigr\}\bigr]^{\frac{k}{k'}}\bigl[P_0\bigl(1-\sqrt{\tfrac{p}{p_0}}\bigr)^2\bigr]^{\frac{k'-k}{k'}}
\end{align*}
by H\"older's inequality with $p=\frac{k'}{k}>1$ and $q=\frac{k'}{k'-k}>1$.
\end{proof}

\begin{rem}
If we want to bound the variation for $1\leq k<2$, we need to impose an assumption to control not only the event $\{\frac{p_0}{p}>4\}$ but also $\{\frac{p_0}{p}\approx 1\}$.
\end{rem}

\section{Comparison} \label{sec:comparison}

We compare condition (\ref{asm:BN}) with four other conditions.
The first one is the {\em uniform boundedness condition}:
\begin{equation} \label{asm:bounded} \tag{UB}
	\biggl\|\frac{p_0}{p}\biggr\|_\infty=\sup_{x\in\mathcal{X}}\,\biggl|\frac{p_0(x)}{p(x)}\biggr|<\infty
\end{equation}
used in \citet[(7.6)]{bm1998}, \citet[Lemma 8.3]{ggv2000}, and \citet[Lemma B.3]{gv2017}.
The second is the condition imposed in \citet[Theorem 5]{ws1995}: for some $\delta\in(0,1]$ and $M<\infty$,
\begin{equation} \label{asm:D} \tag{WS}
	P_0\biggl(\biggl[\frac{p_0}{p}\biggr]^\delta\mathbbm{1}\biggl\{\frac{p_0}{p}>e^{\frac{1}{\delta}}\biggr\}\biggr)\leq M h(p_0,p)^2,
\end{equation}
which is equivalent to (\ref{asm:BN}) but looks weaker as the threshold is made to diverge as $\delta\to 0$.
The third is the {\em finite moment condition}:
\begin{equation} \label{asm:moment} \tag{FM}
	P_0\biggl(\frac{p_0}{p}\biggr)<\infty
\end{equation}
employed in \citet[Lemma 8.7]{ggv2000} and \citet[Lemma B.2]{gv2017}.
The fourth is the {\em finite conditional moment condition}:
\begin{equation} \label{asm:local} \tag{CM}
	M=\inf_{c\geq 1}c P_0\biggl(\frac{p_0}{p}\biggm|\frac{p_0}{p}\geq\biggl[1+\frac{1}{2c}\biggr]^2\biggr)<\infty
\end{equation}
introduced in \citet[Lemma S.4]{kmp2023} and \citet[Lemma 2.1]{kr2023}.
In particular, we show that the following relationship holds:
\[
	\text{(\ref{asm:bounded})}\implies\text{(\ref{asm:local})}\implies
	\begin{array}[t]{@{}c@{}}
	\text{(\ref{asm:BN}, $\delta=1$)}\\
	\rotatebox[origin=c]{90}{\vphantom{$\centernot\Longleftrightarrow$}$\Longleftarrow$}\\
	\clap{\tikzmarknode{mo}{\text{(\ref{asm:moment})}}}
	\end{array}
	\implies
	\tikzmarknode{M}{\text{(\ref{asm:BN}, $\delta<1$)}}
\]
\begin{tikzpicture}[overlay,remember picture]
	\draw[
		double equal sign distance,
		nfold,
		implies-implies,
		line cap=round,
		shorten <=6pt,
		shorten >=6pt,
	] (M.south west) -- node[sloped]{$\centernot{}$} (mo.north east);
\end{tikzpicture}

\vspace{-1.5em}

\subsection{Comparison among Different $\delta$'s} \label{sec:reverse}

For $\delta\leq\delta'$, we obviously have
\(
	P_0([\frac{p_0}{p}]^\delta\mathbbm{1}\{\frac{p_0}{p}>4\})\leq P_0([\frac{p_0}{p}]^{\delta'}\mathbbm{1}\{\frac{p_0}{p}>4\})
\),
so (\ref{asm:BN}) for $\delta'$ implies (\ref{asm:BN}) for $\delta$.
Not only that, we can further show that the multiples grow at different rates.
For example, the following theorem implies that if $P_0(\frac{p_0}{p}\mathbbm{1}\{\frac{p_0}{p}>4\})\leq M h(p_0,p)^2$, then $P_0(\sqrt{\tfrac{p_0}{p}}\mathbbm{1}\{\frac{p_0}{p}>4\})\leq 4\sqrt{M}h(p_0,p)^2$.

\begin{prop}[$\text{(\ref{asm:BN}, $\delta'$)}\Rightarrow\text{(\ref{asm:BN}, $\delta$)}$]
For arbitrary probability measures $P_0$ and $P$ and $0<\delta\leq\delta'$,
\[
	P_0\biggl(\biggl[\frac{p_0}{p}\biggr]^\delta\mathbbm{1}\biggl\{\frac{p_0}{p}>4\biggr\}\biggr)\leq 4 h(p_0,p)^{2(1-\frac{\delta}{\delta'})} \biggl[P_0\biggl(\biggl[\frac{p_0}{p}\biggr]^{\delta'}\mathbbm{1}\biggl\{\frac{p_0}{p}>4\biggr\}\biggr)\biggr]^{\frac{\delta}{\delta'}}.
\]
That is, if (\ref{asm:BN}) holds for $\delta'$, then (\ref{asm:BN}) holds for $\delta$ with a multiple that grows slower if grows at all.
\end{prop}

\begin{proof}
It mirrors \cref{lem:KL3} (\ref{lem:KL3:BTW}).
Replace $\log\tfrac{p_0}{p}$, $k$, and $k'$ with $\tfrac{p_0}{p}$, $\delta$, and $\delta'$.
\end{proof}

\begin{exa}[$\text{(\ref{asm:BN}, $\delta'$)}\nLeftarrow\text{(\ref{asm:BN}, $\delta$)}$] \label{exa:1}
Let $p_0(x)=\mathbbm{1}\{0<x<1\}$ be the uniform density over $(0,1)$, and consider $p(x)=2 x\mathbbm{1}\{0<x<1\}$.
Then, it is easy to verify that (\ref{asm:BN}) for $\delta=1$ fails but (\ref{asm:BN}) for $\delta=1/2$ holds.
\end{exa}

\subsection{Condition in \citet{ws1995}} \label{sec:ws1995}

While (\ref{asm:D}) is equivalent to (\ref{asm:BN}), \citet[Theorem 5]{ws1995} used it to bound the Kullback--Leibler measures rather than the Bernstein ``norm''.
We verify that (\ref{asm:D}) is indeed a sufficient condition for (\ref{asm:KLD}) and (\ref{asm:KLV}).
The next proposition recovers the multiples of the same orders as \citet[Theorem 5]{ws1995} when combined with \cref{lem:KL3} for $k=1,2$.

\begin{prop}[$\text{(\ref{asm:D})}\Rightarrow\text{(\ref{asm:KLD}, \ref{asm:KLV})}$] \label{lem:WS}
If (\ref{asm:D}) holds for $\delta\in(0,1]$ and $M<\infty$, then for every real $k>0$,
\[
	P_0\biggl(\biggl[\log\frac{p_0}{p}\biggr]^k\mathbbm{1}\biggl\{\frac{p_0}{p}>4\biggr\}\biggr)\leq\delta^{-k}\biggl[4+\frac{e}{(\sqrt{e}-1)^2}(k\vee\log M)^k\biggr]h(p_0,p)^2.
\]
\end{prop}

\begin{proof}
For $0<\delta\leq 1$ and $e^{-1/\delta}\leq x<\frac{1}{4}$, we have
\(
	(\log\tfrac{1}{x})^k\leq\tfrac{1}{\delta^k}=\tfrac{4}{\delta^k}(\tfrac{1}{\sqrt{4}}-1)^2\leq\tfrac{4}{\delta^k}(\sqrt{x}-1)^2
\),
noting that for $(\log 4)^{-1}<\delta\leq 1$, there is no $x$ that satisfies $e^{-1/\delta}\leq x<\frac{1}{4}$ so it is vacuously true.
Therefore,
\[
	P_0\bigl(\log\tfrac{p_0}{p}\bigr)^k\mathbbm{1}\bigl\{\tfrac{p_0}{p}>4\bigr\}\leq P_0\bigl(\log\tfrac{p_0}{p}\bigr)^k\mathbbm{1}\bigl\{\tfrac{p_0}{p}>e^{\frac{1}{\delta}}\bigr\}+\tfrac{4}{\delta^k}P_0\bigl(\sqrt{\tfrac{p}{p_0}}-1\bigr)^2.
\]
Observe that for $k>0$, $\eta>0$, and $B>0$,
\begin{equation} \label{min}
	\inf_{0<r\leq\eta}\tfrac{1}{r^k}B^{r}=\begin{cases}\frac{e^k}{k^k}(\log B)^k&\text{if $B>e^{\frac{k}{\eta}}$,}\\ \frac{1}{\eta^k}B^{\eta}&\text{if $B\leq e^{\frac{k}{\eta}}$.}\end{cases}
\end{equation}
Hence, for $0<r\leq 1$ and $0<x<e^{-1/\delta}$, we have
\(
	\tfrac{k^k}{e^k}\tfrac{1}{r^k\delta^k}(\tfrac{1}{x})^{r\delta}
	\geq\tfrac{k^k}{e^k}\inf_{0<r\delta}\tfrac{1}{r^k\delta^k}(\tfrac{1}{x})^{r\delta}
	=(\log\tfrac{1}{x})^k
\).
Now we see that
\begin{align*}
	P_0\bigl(\log\tfrac{p_0}{p}\bigr)^k\mathbbm{1}\bigl\{\tfrac{p_0}{p}>e^{\frac{1}{\delta}}\bigr\}&\leq\tfrac{k^k}{e^k}\tfrac{1}{r^k\delta^k}P_0(\tfrac{p_0}{p})^{r\delta}\mathbbm{1}\bigl\{\tfrac{p_0}{p}>e^{\frac{1}{\delta}}\bigr\}\\
	&\leq\tfrac{k^k}{e^k}\tfrac{1}{r^k\delta^k}P_0(\tfrac{p_0}{p})^{r\delta}\mathbbm{1}\bigl\{\tfrac{p_0}{p}>e^{\frac{1}{\delta}}\bigr\}(1-e^{-1/2})^{-2}\bigl(1-\sqrt{\tfrac{p}{p_0}}\bigr)^{2(1-r)} \tag{for $\delta\leq 1$, $r\leq 1$} \\
	&\leq\tfrac{k^k}{e^k}\tfrac{1}{r^k\delta^k}\tfrac{1}{(1-e^{-1/2})^2}\bigl[P_0\bigl(\tfrac{p_0}{p}\bigr)^{\delta}\mathbbm{1}\bigl\{\tfrac{p_0}{p}>e^{\frac{1}{\delta}}\bigr\}\bigr]^r[h(p_0,p)^2]^{1-r}\tag{H\"older's inequality for $p=\frac{1}{r}\geq 1$, $q=\frac{1}{1-r}>1$}\\
	&\leq\tfrac{k^k}{e^k}\tfrac{1}{r^k\delta^k}\tfrac{e}{(\sqrt{e}-1)^2}M^r h(p_0,p)^2. \tag{by (\ref{asm:D})}
\end{align*}
Since this holds for every $0<r\leq 1$, use (\ref{min}) again to obtain the bound.
\end{proof}

\subsection{Finite Conditional Moment Condition} \label{sec:CM}

Condition (\ref{asm:local}) was introduced in \citet[Lemma S.4]{kmp2023} and \citet[Lemma 2.1]{kr2023} and was the first to accommodate unbounded likelihood ratios for the sharp Hellinger bound on the Bernstein ``norm''.
However, it is stronger than (\ref{asm:BN}) in the following sense:
\begin{enumerate}[(a)]
	\item For fixed $P_0$ and $P$, (\ref{asm:BN}, $\delta=1$) and (\ref{asm:local}) are equivalent.
	\item A universal constant in (\ref{asm:BN}, $\delta=1$) does not imply a universal constant in (\ref{asm:local}).
\end{enumerate}
These are consequences of the following proposition and example.

\begin{prop}[$\text{(\ref{asm:local})}\Rightarrow\text{(\ref{asm:BN}, $\delta=1$)}$] \label{lem:suff}
The following hold.
\begin{enumerate}[(i)]
	\item \label{lem:stuff:1}
		If (\ref{asm:local}) holds, then
		\[
			P_0\biggl(\frac{p_0}{p}\mathbbm{1}\biggl\{\frac{p_0}{p}>4\biggr\}\biggr)\leq(2 M+1)^2 h(p_0,p)^2.
		\]
	\item \label{lem:stuff:2}
		If (\ref{asm:BN}) holds with $\delta=1$, then (\ref{asm:local}) holds but $M$ can be arbitrarily large.
\end{enumerate}
\end{prop}

\begin{proof}
(\ref{lem:stuff:1})
Note that the infimum in (\ref{asm:local}) is attained at some finite $c\leq 1\vee M$.
Denote $C=[1+\frac{1}{2c}]^2$.
Since $C>1$,
\begin{multline*}
	h(p_0,p)^2
	\geq\int(\sqrt{p_0}-\sqrt{p})^2\mathbbm{1}\bigl\{\tfrac{p_0}{p}\geq C\bigr\}
	\geq\int\bigl(\sqrt{p_0}-\sqrt{\tfrac{p_0}{C}}\bigr)^2\mathbbm{1}\bigl\{\tfrac{p_0}{p}\geq C\bigr\} \\
	=\bigl(1-\tfrac{1}{\sqrt{C}}\bigr)^2 P_0\bigl(\tfrac{p_0}{p}\geq C\bigr)
	=\tfrac{1}{(2c+1)^2}P_0\bigl(\tfrac{p_0}{p}\geq C\bigr).
\end{multline*}
This implies
\[
	P_0\tfrac{p_0}{p}\mathbbm{1}\bigl\{\tfrac{p_0}{p}>4\bigr\}
	\leq P_0\tfrac{p_0}{p}\mathbbm{1}\bigl\{\tfrac{p_0}{p}\geq C\bigr\}
	=P_0\bigl(\tfrac{p_0}{p}\bigm|\tfrac{p_0}{p}\geq C\bigr)P_0\bigl(\tfrac{p_0}{p}\geq C\bigr)
	\leq\tfrac{(2c+1)^2}{c}M h(p_0,p)^2.
\]
Use $1\leq c\leq 1\vee M$ to complete the proof.

(\ref{lem:stuff:2})
Trivially, if (\ref{asm:BN}) holds for $\delta=1$, then (\ref{asm:local}) holds with some $M<\infty$.
\cref{exa:doom} demonstrates that $M$ can be arbitrarily large.
\end{proof}

\begin{exa}[$\text{(\ref{asm:local})}\nLeftarrow\text{(\ref{asm:BN}, $\delta=1$)}$] \label{exa:doom}
Let $p_0$ be the uniform density over $(0,1)$. %
For $\theta\in[0,1/4)$, consider the model
\[
	p_\theta(x)=\begin{cases}\theta&0<x<\theta^2,\\1-\theta&\theta^2\leq x<1-\theta,\\\frac{1-\theta^3-(1-\theta)(1-\theta-\theta^2)}{\theta}&1-\theta\leq x<1.\end{cases}
\]
Then, $h(p_0,p_\theta)^2$ is approximately linear at $\theta=0$ with slope $3-2\sqrt{2}$ and $P_0(\frac{p_0}{p_\theta}\mathbbm{1}\{\frac{p_0}{p_\theta}>4\})=\theta$.
Therefore, (\ref{asm:BN}) is satisfied for $\delta=1$ and the multiple $M=(3-2\sqrt{2})^{-1}$.
Meanwhile, if we pick $c$ such that $\frac{1}{1-\theta}<[1+\frac{1}{2c}]^2$, then%
\footnote{Since $c\geq 1$ and $\theta<1/4$, $[1+\frac{1}{2c}]^2<\frac{1}{\theta}$ is granted.}
\[
	P_0\biggl(\frac{p_0}{p_\theta}\biggm|\frac{p_0}{p_\theta}\geq\biggl[1+\frac{1}{2c}\biggr]^2\biggr)=\frac{1}{\theta}\conv\infty \qquad \text{as}\quad\theta\to 0.
\]
If we pick $c$ such that $\frac{1}{1-\theta}\geq[1+\frac{1}{2c}]^2$, then %
\[
	c P_0\biggl(\frac{p_0}{p_\theta}\biggm|\frac{p_0}{p_\theta}\geq\biggl[1+\frac{1}{2c}\biggr]^2\biggr)\geq\frac{1}{2}\frac{\sqrt{1-\theta}}{1-\sqrt{1-\theta}}\cdot\frac{\frac{1}{\theta}\theta^2+\frac{1}{1-\theta}(1-\theta-\theta^2)}{1-\theta}\conv\infty.
\]
Therefore, $M$ in (\ref{asm:local}) can be arbitrarily large.
\end{exa}

This gain in generality is due to the localization by an unconditional expectation.
When we want to achieve $P_0(\sqrt{\frac{p_0}{p}}-1)^2\lesssim h(p_0,p)^2$, it does not have to be that
\[
	P_0\biggl(\sqrt{\frac{p_0}{p}}-1\biggr)^2\mathbbm{1}\biggl\{\frac{p_0}{p}>C\biggr\}\lesssim\int(\sqrt{p_0}-\sqrt{p})^2\mathbbm{1}\biggl\{\frac{p_0}{p}>C\biggr\}
\]
event by event.
Condition (\ref{asm:local}) requires that this hold for the local unfavorable event $1<C\leq\frac{9}{4}$, while condition (\ref{asm:BN}) allows for the possibility that
\[
	P_0\biggl(\sqrt{\frac{p_0}{p}}-1\biggr)^2\mathbbm{1}\biggl\{\frac{p_0}{p}>4\biggr\}\lesssim\int(\sqrt{p_0}-\sqrt{p})^2\mathbbm{1}\biggl\{\frac{p_0}{p}\leq 4\biggr\}.
\]

\subsection{Uniform Boundedness Condition} \label{sec:UB}

Condition (\ref{asm:bounded}), which was used in \citet[(7.6)]{bm1998}, \citet[Lemma 8.3]{ggv2000}, and \citet[Lemma B.3]{gv2017}, requires that the likelihood ratio be uniformly bounded on the entire support.
This condition is arguably restrictive, especially when the underlying random variables are unbounded, since then $p_0$ and $p$ can take arbitrarily small values.

It trivially implies the finite conditional moment, hence all other conditions.

\begin{prop}[$\text{(\ref{asm:bounded})}\Rightarrow\text{(\ref{asm:local})}$]
If (\ref{asm:bounded}) holds, then (\ref{asm:local}) holds for $M\leq\lVert\frac{p_0}{p}\rVert_\infty$.
\end{prop}

\begin{proof}
$M=\inf_{c\geq 1}c P_0(\frac{p_0}{p}\mid\frac{p_0}{p}\geq[1+\frac{1}{2c}]^2)\leq P_0(\frac{p_0}{p}\mid\frac{p_0}{p}\geq\frac{9}{4})\leq\lVert\frac{p_0}{p}\rVert_\infty$.
\end{proof}

This example shows that the normal location model satisfies (\ref{asm:local}) but not (\ref{asm:bounded}).

\begin{exa}[$\text{(\ref{asm:bounded})}\nLeftarrow\text{(\ref{asm:local})}$; Normal location model] \label{exa:normal}
Let $\mathcal{P}=\{P_\theta=N(\theta,1):\theta\in\mathbb{R}\}$ and $P_0=N(0,1)$.
Then, we have
\[
	\frac{p_0(x)}{p_\theta(x)}=\exp\biggl(-\theta x+\frac{\theta^2}{2}\biggr).
\]
Hence, $\lVert\frac{p_0}{p_\theta}\rVert_\infty$ is infinity for $\theta\neq 0$, failing (\ref{asm:bounded}).
Meanwhile,
\[
	\inf_{c\geq 1}c P_0\biggl(\frac{p_0}{p_\theta}\biggm|\frac{p_0}{p_\theta}\geq\biggl[1+\frac{1}{2c}\biggr]^2\biggr)
	\leq P_0\biggl(\frac{p_0}{p_\theta}\biggm|\frac{p_0}{p_\theta}\geq e\biggr)
	=\frac{e^{\theta^2}\Phi(\frac{\lvert\theta\rvert}{2}-\frac{1}{\lvert\theta\rvert}+\lvert\theta\rvert)}{\Phi(\frac{\lvert\theta\rvert}{2}-\frac{1}{\lvert\theta\rvert})}.
\]
This is finite for every $\theta$ (it approaches $1$ as $\theta\to 0$).
Therefore, (\ref{asm:local}) holds on every bounded set of $\theta$.
\end{exa}

\subsection{Finite Moment Condition} \label{sec:moment}

Condition (\ref{asm:moment}) was employed in \citet[Lemma 8.7]{ggv2000} and \citet[Lemma B.2]{gv2017} to derive Hellinger bounds whose multiples diverge as the Hellinger distance converges to zero.

Conditions (\ref{asm:moment}) and (\ref{asm:BN}, $\delta<1$) do not imply each other, while each of them is implied by (\ref{asm:BN}, $\delta=1$).

\begin{exa}[$\text{(\ref{asm:moment})}\nLeftarrow\text{(\ref{asm:BN}, $\delta<1$)}$]
In \cref{exa:1}, we have $P_0(\frac{p_0}{p})=\infty$ while (\ref{asm:BN}, $\delta=1/2$) holds.
Thus, (\ref{asm:moment}) is not even necessary for the Hellinger dominance.
\end{exa}

The next example shows that (\ref{asm:BN}, $\delta<1$) is not implied by (\ref{asm:moment}), and hence neither is (\ref{asm:BN}, $\delta=1$) implied by (\ref{asm:moment}).

\begin{exa}[$\text{(\ref{asm:moment})}\nRightarrow\text{(\ref{asm:BN}, $\delta<1$)}$] \label{exa:counter}
First, we show that $\sup_p P_0(\frac{p_0}{p})<\infty$ does not imply (\ref{asm:BN}) for $\delta<1$ with a universal constant $M<\infty$.
Let $p_0(x)=\mathbbm{1}\{0<x<1\}$ be the uniform density over $(0,1)$, and consider the model
\[
	p_\theta(x)=\begin{cases}\theta&0<x\leq\theta,\\1+\theta&\theta<x<1,\end{cases}
\]
indexed by $\theta\in[0,1/4)$.
Hence, $p_\theta$ equals $p_0$ when $\theta=0$.
This model satisfies (\ref{asm:moment}) since
\[
	P_0\biggl(\frac{p_0}{p_\theta}\biggr)=1+\frac{1-\theta}{1+\theta}\leq 2.
\]
Meanwhile, we have
\begin{gather*}
	P_0\biggl(\sqrt{\frac{p_0}{p_\theta}}\mathbbm{1}\biggl\{\frac{p_0}{p_\theta}>4\biggr\}\biggr)=\sqrt{\theta},\\
	h(p_0,p_\theta)^2=\theta(1-\sqrt{\theta})^2+(1+\theta)(1-\sqrt{1+\theta})^2\approx\theta.
\end{gather*}
Thus, (\ref{asm:BN}, $\delta=1/2$) fails along $\theta\to 0$.
We can further verify that Hellinger dominance fails as $\theta\to 0$.
\end{exa}

We have already shown that (\ref{asm:BN}, $\delta<1$) is implied by (\ref{asm:BN}, $\delta=1$) in \cref{sec:reverse}, so what remains to be proved is that (\ref{asm:moment}) is implied by (\ref{asm:BN}, $\delta=1$).

\begin{prop}[$\text{(\ref{asm:BN}, $\delta=1$)}\Rightarrow\text{(\ref{asm:moment})}$]
If (\ref{asm:BN}) holds for $\delta=1$, then (\ref{asm:moment}) holds.
\end{prop}

\begin{proof}
By the Cauchy--Schwarz inequality,
\begin{align*}
	P_0\tfrac{p_0}{p}
	&=P_0\tfrac{p_0}{p}\mathbbm{1}\bigl\{\tfrac{p_0}{p}>4\bigr\}+P_0\bigl(1-\sqrt{\tfrac{p}{p_0}}\bigr)\sqrt{\tfrac{p_0}{p}}\bigl(\sqrt{\tfrac{p_0}{p}}+1\bigr)\mathbbm{1}\bigl\{\tfrac{p_0}{p}\leq 4\bigr\}+P_0\bigl(\tfrac{p_0}{p}\leq 4\bigr)\\
	&\leq P_0\tfrac{p_0}{p}\mathbbm{1}\bigl\{\tfrac{p_0}{p}>4\bigr\}+6 \sqrt{P_0\bigl(1-\sqrt{\tfrac{p}{p_0}}\bigr)^2}+1
	\leq P_0\tfrac{p_0}{p}\mathbbm{1}\bigl\{\tfrac{p_0}{p}>4\bigr\}+6 h(p_0,p)+1,
\end{align*}
which is finite under (\ref{asm:BN}, $\delta=1$) since $h(p_0,p)^2\leq 2$ by construction.
\end{proof}

\section{Application to Maximum Likelihood Estimation} \label{sec:mle}

We now apply \cref{thm:BN,lem:KL3} to relax the bounded likelihood ratio condition in nonparametric maximum likelihood estimation, namely \citet[Theorem 3.4.12]{vw2023}.
To control the complexity of a class of functions, we use the bracketing integral.

\begin{defn}[Bracketing number and integral]
Let $d$ be a premetric on real\hyp{}valued functions that is compatible with pointwise partial ordering \citep[p.~528]{gv2017}, that is, (i) $d(f,f)=0$, (ii) $d(f,g)=d(g,f)\geq 0$, and (iii) $d(l,u)=\sup\{d(f,g):l\leq f,g\leq u\}$ for every $f,g,l,u$.
A pair of functions $[l,u]$ is called an {\em $\varepsilon$\hyp{}bracket} if $\ell\leq u$ and $d(l,u)<\varepsilon$.
The {\em bracketing number} $N_{[]}(\varepsilon,\mathcal{F},d)$ is the minimum number of $\varepsilon$\hyp{}brackets needed to cover a set of functions $\mathcal{F}$.%
\footnote{The bracketing functions need not come from $\mathcal{F}$ but are confined to the function space defined by $d$. This means that if $d$ is (derived from) a norm, the brackets need to have finite norms \citep[Definition 2.1.6]{vw2023}.}
The {\em nonstandardized bracketing integral} is defined by
\[
	\tilde{J}_{[]}(\delta,\mathcal{F},d)\coloneqq\int_0^\delta\sqrt{1+\log N_{[]}(\varepsilon,\mathcal{F},d)}d\varepsilon.
\]
\end{defn}

The bracketing integral is an increasing and concave function in $\delta$; hence, for $c\geq 1$, we have $\tilde{J}_{[]}(c\delta,\mathcal{F},d)\leq c\tilde{J}_{[]}(\delta,\mathcal{F},d)$.

The next theorem generalizes the sieve maximum likelihood theorem of \citet[Theorem 3.4.12]{vw2023} to models with unbounded likelihood ratios.%
\footnote{Another similar result is \citet[Theorem F.4]{gv2017}.
This is a special case of \citet[Theorem 3.4.12]{vw2023} with an additional assumption that $p_0\in\mathcal{P}_n$, so one can always take $p_n=p_0$ in the theorem statement.}
Let $\mathbb{P}_n$ denote the empirical measure of an i.i.d.~sample $X_1,X_2,\dots,X_n$.

\begin{thm}[Rate of convergence of sieve MLE] \label{thm:mle}
Let $X_1,X_2,\dots$ be an independent sequence from a probability distribution $P_0$.
Let $\mathcal{P}_n$ be a sequence of arbitrary sets of probability distributions, and denote by $\mathcal{P}_{n,\delta}=\{p\in\mathcal{P}_n:h(p_0,p)\leq\delta\}$ the $\delta$\hyp{}neighborhood of $p_0$ with respect to the Hellinger distance.
Suppose there exist sequences $p_n\in\mathcal{P}_n$ and $\delta_n\geq 0$ and $M\in[0,\infty)$ that satisfy the following three conditions: %
\begin{gather}
	h(p_0,p_n)\lesssim\delta_n, \label{asm:delta} \\
	P_0\biggl(\biggl[\log\frac{p_0}{p_n}\biggr]^2\mathbbm{1}\biggl\{\frac{p_0}{p_n}>4\biggr\}\biggr)\leq M \delta_n^2, \label{asm:Mn} \\
	\tilde{J}_{[]}(\delta_n,\mathcal{P}_{n,\delta_n},h)\leq\delta_n^2\sqrt{n}. \label{asm:entropy}
\end{gather}
Then, the approximate maximizer $\hat{p}_n\in\mathcal{P}_n$ of the likelihood $p\mapsto\prod_{i=1}^n p(X_i)$ in the sense that
\begin{equation}
	\mathbb{P}_n\log\hat{p}_n\geq\mathbb{P}_n\log p_n-O_P(\delta_n^2) \label{asm:max}
\end{equation}
satisfies $h(p_0,\hat{p}_n)=O_P(\delta_n\vee n^{-1/2})$.
\end{thm}

\begin{rem}
Condition (\ref{asm:delta}) quantifies the sieve's approximation power in terms of the Hellinger distance, which is equivalent to the Kullback--Leibler divergence thanks to condition (\ref{asm:Mn}) and \cref{lem:KL3}.
A sufficient condition for (\ref{asm:Mn}) is that $p_0/p_n$ is uniformly bounded.
Condition (\ref{asm:entropy}) controls the local entropy of the sieve, for which a sufficient condition is the restriction of the global entropy, $\tilde{J}_{[]}(\delta_n,\mathcal{P}_n,h)\leq\delta_n^2\sqrt{n}$.
Condition (\ref{asm:max}) requires that the optimization algorithm does as good as $p_n\in\mathcal{P}_n$, so a sufficient condition is to pin down the global maximizer within a tolerance, $\mathbb{P}_n\log\hat{p}_n\geq\sup_{p\in\mathcal{P}_n}\mathbb{P}_n\log p-O_P(\delta_n^2)$.
With these sufficient conditions, \cref{thm:mle} reduces to \citet[Theorem 3.4.12]{vw2023}.
\end{rem}

\begin{rem}
It is straightforward to replace $M$ in (\ref{asm:Mn}) with $M_n\to\infty$ and obtain the rate $\delta_n\sqrt{M_n}$.
Then, replace also $\delta_n$ in (\ref{asm:entropy}) and (\ref{asm:max}) by $\delta_n\sqrt{M_n}$.
Note that even when (\ref{asm:bounded}) holds, this $M_n$ grows logarithmically slower than the multiple that arises from \citet[Lemma 3.4.10]{vw2023}.%
\footnote{To see this, use similar arguments as those in \cref{sec:ws1995,sec:CM,sec:UB}.}
\end{rem}

\begin{proof}
To exploit the trick introduced by \citet{bm1993} \citep[see][Section 3.4.4]{vw2023}, let $m_p^q\coloneqq\log\frac{p+q}{2p}$.

We will apply \citet[Theorem 3.4.1]{vw2023} to prove this theorem, where the mapping between their notation (LHS) and ours (RHS) is
\begin{align*}
	\theta&=p,\\
	\theta_{n,0}&=p_0,\\
	\theta_n&=p_n,\\
	\Theta_n&=\mathcal{P}_n,\\
	d_n(\theta,\theta_{n,0})&=h(p_0,p),\\
	\ubar{\delta}_n&=h(p_0,p_n),\\
	\delta_n&=2(\sqrt{6}+\sqrt{3})\ubar{\delta}_n,\\
	\mathbb{M}_n(\theta)&=(4+2\sqrt{2})^2\mathbb{P}_n m_{p_0}^p,\\ %
	M_n(\theta)&=(4+2\sqrt{2})^2 P_0 m_{p_0}^p,\\ %
	\phi_n(\delta)&=\text{an appropriate majorant of }\tilde{J}_{[]}(\delta,\mathcal{P}_{n,\delta},h)\Bigl[1+\tfrac{\tilde{J}_{[]}(\delta,\mathcal{P}_{n,\delta},h)}{\delta^2\sqrt{n}}\Bigr].
\end{align*}
A key departure from \citet[Theorem 3.4.12]{vw2023} is that we do not let $\theta_{n,0}=p_n$ but set $\theta_{n,0}=p_0$.
This saves us unnecessary complications that arise from dealing with ``misspecification'' of $\mathcal{P}_n$.

We begin by introducing a useful inequality for later use.
Since $(1-\frac{1}{\sqrt{2}})^2(\sqrt{p_0}-\sqrt{p})^2\leq(\sqrt{p_0}-\sqrt{\frac{p_0+p}{2}})^2\leq\frac{1}{2}(\sqrt{p_0}-\sqrt{p})^2$, we have
\begin{equation} \label{eq:halfhellinger}
	\bigl(1-\tfrac{1}{\sqrt{2}}\bigr)^2 h(p_0,p)^2\leq h\bigl(p_0,\tfrac{p_0+p}{2}\bigr)^2\leq\tfrac{1}{2}h(p_0,p)^2.
\end{equation}
Second, since $0\leq\frac{2p_0}{p_0+p}\leq 2$, when we compare $\frac{p_0+p}{2}$ against $p$, the LHS of (\ref{asm:BN}) is zero (so the multiple is zero).

Now, we verify each condition of \citet[Theorem 3.4.1]{vw2023}.
The first condition is that for every $n$ and $\delta>\ubar{\delta}_n$,
\[
	\sup_{p\in\mathcal{P}_n:\frac{\delta}{2}<h(p_0,p)\leq\delta}(4+2\sqrt{2})^2\bigl(P_0 m_{p_0}^p-P_0 m_{p_0}^{p_0}\bigr)\leq-\delta^2.
\]
This immediately follows from \cref{thm:BN} (\ref{thm:BN:KD}) and (\ref{eq:halfhellinger}) since %
\begin{multline*}
	P_0 m_{p_0}^p-P_0 m_{p_0}^{p_0}=-P_0\log\tfrac{2p_0}{p_0+p}\\
	\leq-h\bigl(p_0,\tfrac{p_0+p}{2}\bigr)^2
	\leq-\bigl(1-\tfrac{1}{\sqrt{2}}\bigr)^2 h(p_0,p)^2
	\leq-\bigl(1-\tfrac{1}{\sqrt{2}}\bigr)^2\tfrac{\delta^2}{4}
	=-\tfrac{\delta^2}{(4+2\sqrt{2})^2}.
\end{multline*}
The second condition to establish is
\[
	\mathbb{E}^*\sup_{p\in\mathcal{P}_n:h(p_0,p)\leq\delta}(4+2\sqrt{2})^2\sqrt{n}\bigl|(\mathbb{P}_n-P_0)m_{p_0}^p-(\mathbb{P}_n-P_0)m_{p_0}^{p_0}\bigr|\lesssim\phi_n(\delta).
\]
Note that \cref{thm:BN} and (\ref{eq:halfhellinger}) yield
\[
	\lVert m_{p_0}^p\rVert_{P_0,B}=\bigl\|\log\tfrac{2p_0}{p_0+p}\bigr\|_{P_0,B}\leq\sqrt{18}h\bigl(p_0,\tfrac{p_0+p}{2}\bigr)\leq 3\delta.
\]
Thus, in light of \citet[Theorem 2.14.18$'$]{vw2023}, we have
\begin{multline*}
	\mathbb{E}^*\sup_{p\in\mathcal{P}_n:h(p_0,p)\leq\delta}\sqrt{n}\bigl|(\mathbb{P}_n-P_0)m_{p_0}^p-(\mathbb{P}_n-P_0)m_{p_0}^{p_0}\bigr|\\
	\lesssim\tilde{J}_{[]}(3\delta,\mathcal{M}_{n,\delta},\lVert\cdot\rVert_{P_0,B})\Bigl[1+\tfrac{\tilde{J}_{[]}(3\delta,\mathcal{M}_{n,\delta},\lVert\cdot\rVert_{P_0,B})}{9\delta^2\sqrt{n}}\Bigr],
\end{multline*}
where $\mathcal{M}_{n,\delta}=\{m_{p_0}^p:p\in\mathcal{P}_n,h(p_0,p)\leq\delta\}$.
Let $[\ell,u]$ be an $\varepsilon$\hyp{}bracket in $\mathcal{P}$ with respect to $h$.
Since $u\geq\ell$ and $e^{\lvert x\rvert}-1-\lvert x\rvert\leq 2(e^{x/2}-1)^2$ for $x\geq 0$, we have
\[
	\lVert m_{p_0}^u-m_{p_0}^\ell\bigr\|_{P_0,B}^2
	\leq 4 P_0\bigl(\sqrt{\tfrac{p_0+u}{p_0+\ell}}-1\bigr)^2
	\leq 4\int\bigl(\sqrt{p_0+u}-\sqrt{p_0+\ell}\bigr)^2\leq 4 h(u,\ell)^2. %
\]
Thus, $[m_{p_0}^\ell,m_{p_0}^u]$ makes a $2\varepsilon$\hyp{}bracket in $\mathcal{M}$ with respect to the Bernstein ``norm'', so
\(
	N_{[]}(2\varepsilon,\mathcal{M}_{n,\delta},\lVert\cdot\rVert_{P_0,B})\leq N_{[]}(\varepsilon,\mathcal{P}_{n,\delta},h)
\).
This implies
\begin{align*}
	\tilde{J}_{[]}(3\delta,\mathcal{M}_{n,\delta},\lVert\cdot\rVert_{P_0,B})
	&=\int_0^{3\delta}\sqrt{1+\log N_{[]}(\varepsilon,\mathcal{M}_{n,\delta},\lVert\cdot\rVert_{P_0,B})}d\varepsilon \\
	&\leq\int_0^{3\delta}\sqrt{1+\log N_{[]}(\tfrac{\varepsilon}{2},\mathcal{P}_{n,\delta},h)}d\varepsilon \\
	&=2\int_0^{\frac{3}{2}\delta}\sqrt{1+\log N_{[]}(\varepsilon,\mathcal{P}_{n,\delta},h)}d\varepsilon \\
	&=2\tilde{J}_{[]}(\tfrac{3}{2}\delta,\mathcal{P}_{n,\delta},h)
	\leq 3\tilde{J}_{[]}(\delta,\mathcal{P}_{n,\delta},h).
\end{align*}
Thus, we obtain
\[
	\mathbb{E}^*\sup_{p\in\mathcal{P}_{n,\delta}}\sqrt{n}\bigl|(\mathbb{P}_n-P_0)m_{p_0}^p-(\mathbb{P}_n-P_0)m_{p_0}^{p_0}\bigr|
	\lesssim\tilde{J}_{[]}(\delta,\mathcal{P}_{n,\delta},h)\Bigl[1+\tfrac{\tilde{J}_{[]}(\delta,\mathcal{P}_{n,\delta},h)}{\delta^2\sqrt{n}}\Bigr].
\]
Since $\tilde{J}_{[]}$ is increasing and concave in $\delta$, the RHS has a majorant $\phi_n(\delta)$ that is increasing in $\delta\geq\ubar{\delta}_n$ and, if divided by $\delta^\alpha$ for $1<\alpha<2$, is decreasing in $\delta$.

We now check the properties of $p_n$ and $\delta_n$ required by \citet[Theorem 3.4.1]{vw2023}.
By (\ref{asm:entropy}), we have
\[
	\tilde{J}_{[]}(\delta_n,\mathcal{P}_{n,\delta_n},h)\Bigl[1+\tfrac{\tilde{J}_{[]}(\delta_n,\mathcal{P}_{n,\delta_n},h)}{\delta_n^2\sqrt{n}}\Bigr]\leq 2\delta_n^2\sqrt{n}.
\]
Therefore, $\phi_n$ can be made to satisfy $\phi_n(\delta_n)\leq\delta_n^2\sqrt{n}$.
Next, we show that
\[
	\delta_n^2\geq(4+2\sqrt{2})^2\bigl(P_0 m_{p_0}^{p_0}-P_0 m_{p_0}^{p_n}\bigr).
\]
This can be verified using \cref{lem:KL3} (\ref{lem:KL3:KD}) and (\ref{eq:halfhellinger}),
\[
	P_0 m_{p_0}^{p_0}-P_0 m_{p_0}^{p_n}
	=P_0\log\tfrac{2p_0}{p_0+p_n}
	\leq 3 h\bigl(p_0,\tfrac{p_0+p_n}{2}\bigr)^2
	\leq \tfrac{3}{2} h(p_0,p_n)^2.
\]
Hence, by letting $\delta_n=(4+2\sqrt{2})\sqrt{3/2}\ubar{\delta}_n=2(\sqrt{6}+\sqrt{3})\ubar{\delta}_n$, we have $\delta_n^2\geq(4+2\sqrt{2})^2(P_0 m_{p_0}^{p_0}-P_0 m_{p_0}^{p_n})$ as well as $\delta_n\geq\ubar{\delta}_n$.

Finally, it remains to show that
\(
	\mathbb{P}_n m_{p_0}^{\hat{p}_n}\geq\mathbb{P}_n m_{p_0}^{p_n}-O_P(\delta_n^2)
\).
By the convexity of the logarithm,
\[
	2\mathbb{P}_n m_{p_0}^{\hat{p}_n}
	=2\mathbb{P}_n\log\tfrac{p_0+\hat{p}_n}{2p_0}
	\geq\mathbb{P}_n\log\tfrac{p_0}{p_0}+\mathbb{P}_n\log\tfrac{\hat{p}_n}{p_0}
	\geq\mathbb{P}_n\log\tfrac{p_n}{p_0}+\mathbb{P}_n\log\tfrac{\hat{p}_n}{p_n}.
\]
The second term is bounded from below by $-O_P(\delta_n^2)$ by (\ref{asm:max}).
The first term can be decomposed as
\(
	\mathbb{P}_n\log\tfrac{p_n}{p_0}=-(\mathbb{P}_n-P_0)\log\tfrac{p_0}{p_n}-P_0\log\tfrac{p_0}{p_n}
\).
Combined together, it suffices to show that
\[
	-\tfrac{1}{2}(\mathbb{P}_n-P_0)\log\tfrac{p_0}{p_n}-\tfrac{1}{2}P_0\log\tfrac{p_0}{p_n}+(\mathbb{P}_n-P_0)\log\tfrac{2p_0}{p_0+p_n}+P_0\log\tfrac{2p_0}{p_0+p_n}\geq-O_P(\delta_n^2).
\]
For the nonrandom terms, we apply \cref{lem:KL3} (\ref{lem:KL3:KD}) and (\ref{lem:KL3:BTW}) for $k=2$, (\ref{asm:Mn}), (\ref{eq:halfhellinger}), and (\ref{asm:delta}) to obtain
\begin{align*}
	-\tfrac{1}{2}P_0\log\tfrac{p_0}{p_n}+P_0\log\tfrac{2p_0}{p_0+p_n}&\geq-\tfrac{1}{2}\bigl(3+4\sqrt{M}\bigr)\delta_n^2+h\bigl(p_0,\tfrac{p_0+p_n}{2}\bigr)^2\\
	&\geq\bigl[-\tfrac{3}{2}-2\sqrt{M}+\bigl(1-\tfrac{1}{\sqrt{2}}\bigr)^2\bigr]\delta_n^2\geq-O_P(\delta_n^2).
\end{align*}
To bound the remaining random terms,
observe that, since $\{X_i\}$ is a random sample,
\[
	\Var\bigl((\mathbb{P}_n-P_0)\log\tfrac{p_0}{p_n}\bigr)=\tfrac{1}{n}V_{2,0}(p_0\midd p_n)\leq\tfrac{1}{n}V_2(p_0\midd p_n)\leq\tfrac{8+M}{n}\delta_n^2,
\]
where the last inequality follows from \cref{lem:KL3} (\ref{lem:KL3:KV}) for $k=2$, (\ref{asm:delta}), and (\ref{asm:Mn}).
Thus,
\[
	(\mathbb{P}_n-P_0)\log\tfrac{p_0}{p_n}=O_P\bigl(\tfrac{\delta_n}{\sqrt{n}}\bigr).
\]
(The fact that this is not $O_P(\delta_n^2)$ makes the rate of convergence at least as slow as $n^{-1/2}$.)
The same argument %
implies $(\mathbb{P}_n-P_0)\log\tfrac{2p_0}{p_0+p_n}=O_P(\delta_n/\sqrt{n})$.
We have thus verified all assumptions of \citet[Theorem 3.4.1]{vw2023}.
\end{proof}

The new regularity condition imposed in \cref{thm:mle} is (\ref{asm:Mn}).
This is of course weaker than uniform boundedness, but note also that it is based on (\ref{asm:KLV}), not on (\ref{asm:BN}), even though the proof uses the Bernstein ``norm''.
This is thanks to the trick introduced by \citet{bm1993} and to setting $\theta_{n,0}=p_0$, unlike $\theta_{n,0}=p_n$ in \citet[Theorem 3.4.12]{vw2023}.%
\footnote{It is possible to further relax (\ref{asm:Mn}) by using a more general weak law of large numbers such as \citet{g1992w} when bounding $(\mathbb{P}_n-P_0)\log\frac{p_0}{p_n}$.}

The next example contrasts \cref{thm:mle} with \citet[Theorem 3.4.12]{vw2023}.
To keep it simple, we contrive a ``sieve'' version of the normal location model that is covered by the former but not by the latter.

\begin{exa}[Sieve normal location model]
Let $\mathcal{P}=\{P_\theta=N(\theta,1):\theta\in\mathbb{R}\}$ and $P_0=N(0,1)$.
Consider a sieve $\mathcal{P}_n=\{p_\theta\in\mathcal{P}:\lvert\theta\rvert\geq 1/\sqrt{n}\}$, intentionally excluding $p_0$.
As discussed in \cref{exa:normal}, this model does not have bounded likelihood ratios and hence falls outside the scope of \citet[Theorem 3.4.12]{vw2023}.
To establish the parametric rate, we set $\theta_n=1/\sqrt{n}$, $p_n=p_{\theta_n}$, and $\delta_n\sim 1/\sqrt{n}$, and check the assumptions of \cref{thm:mle}.
Observe that for the normal location model,
\[
	h(p_{\theta_1},p_{\theta_2})^2=2-2 e^{-\frac{(\theta_1-\theta_2)^2}{8}}. %
\]
Thus, we have $h(p_0,p_n)=O(n^{-1/2})$, and (\ref{asm:delta}) is satisfied.
In \cref{exa:normal}, it is verified that the normal location model satisfies (\ref{asm:local}), and since (\ref{asm:local}) implies (\ref{asm:KLV}), it satisfies (\ref{asm:Mn}).
Since (\ref{asm:max}) is satisfied by the maximum likelihood estimator, it remains to show (\ref{asm:entropy}).
For $\delta_n\sim n^{-1/2}$, it boils down to
\[
	\tilde{J}_{[]}(\delta,\mathcal{P}_{n,\delta},h)=O(\delta),
\]
or equivalently, that $N_{[]}(\delta,\mathcal{P}_{n,\delta},h)$ is uniformly bounded over $\delta>0$.
The closed\hyp{}form expression of the Hellinger distance implies that $h(p_0,p_\theta)$ is %
 first\hyp{}order approximated by $\lvert\theta\rvert/2$ around $\theta=0$.
Therefore, in the neighborhood of $\theta=0$, we may interchangeably use $\lvert\theta\rvert\leq\delta$ and $h(p_0,p_\theta)\leq\delta/2$.

Let $[\ell,u]\subset\mathbb{R}$ be a bracket in $\mathbb{R}$ and consider a bracket $[p_L,p_U]$ in $\mathcal{P}$ of the form
\[
	p_L(x)=\inf_{\theta\in[\ell,u]}p_\theta(x)
	=\begin{cases}
	p_u(x)&x<\frac{u-\ell}{2},\\
	p_\ell(x)&x\geq\frac{u-\ell}{2},
	\end{cases}
	\quad
	p_U(x)=\sup_{\theta\in[\ell,u]}p_\theta(x)
	=\begin{cases}
	p_\ell(x)&x<\ell,\\
	p_0(0)&x\in[\ell,u],\\
	p_u(x)&x>u.
	\end{cases}
\]
Note that this bracket contains every distribution $p_\theta$ for $\ell\leq\theta\leq u$.
Since $N_{[]}(\delta,\{\theta\in\mathbb{R}:\lvert\theta\rvert\leq\delta\},\lvert\cdot\rvert)$ is independent of $\delta$ (namely, $1$), if we show that $h(p_U,p_L)=O(u-\ell)$, then $N_{[]}(\delta,\mathcal{P}_{n,\delta},h)$ can be bounded by a constant independent of $\delta$, establishing (\ref{asm:entropy}).
Observe that
\begin{align*}
	h(p_U,p_L)^2 %
	&=\int_{-\infty}^\infty\bigl(\sqrt{p_u}-\sqrt{p_\ell}\bigr)^2+\int_\ell^u\bigl[\bigl(\sqrt{p_0(0)}-\sqrt{p_L}\bigr)^2-\bigl(\sqrt{p_U}-\sqrt{p_L}\bigr)^2\bigr]\\
	&\leq h(p_u,p_\ell)^2+\int_\ell^u\bigl(\sqrt{p_0(0)}-\sqrt{p_u(\ell)}\bigr)^2
	=h(p_u,p_\ell)^2+o((u-\ell)^2).
\end{align*}
Thus, we have shown that $h(p_U,p_L)=O(u-\ell)$.
Therefore, \cref{thm:mle} implies that the maximum likelihood estimator converges at rate $n^{-1/2}$.
\end{exa}

Although this example is deliberately simple, it demonstrates the capacity of \cref{thm:mle} to deliver sharp rates in models that escaped previously discussed conditions.

Additional applications of \cref{thm:BN,lem:KL3} arise in the study of posterior contraction rates in nonparametric Bayesian inference. As noted by \citet[p.~199]{gv2017}, once the equivalence between the Hellinger distance and the Kullback--Leibler divergence and variation distance is established, \citet[Theorem 8.9]{gv2017} can be reformulated solely in terms of the Hellinger distance. A further application is the relaxation of the boundedness assumption in \citet[Lemma 9.4 (iii)]{gv2017}. Finally, \cref{thm:BN,lem:KL3} may also help connect minimax convergence rates under different loss functions; see, for example, that similar inequalities are being used in \citet[Lemma 4.4]{b1983}, \citet[Lemma 14]{bbm1999}, and \citet[Lemma 2]{yb1999}.

\section{Conclusion} \label{sec:conc}

We established sharp Hellinger dominance for likelihood-based discrepancy measures under minimal moment conditions.
In particular, we developed the necessary and sufficient conditions for the Hellinger bounds over the fractional Bernstein ``norm'' of the log-likelihood ratio (\cref{thm:BN}), the Kullback--Leibler divergence (\cref{lem:KL3} (\ref{lem:KL3:KD})), and the Kullback--Leibler variation (\cref{lem:KL3} (\ref{lem:KL3:KV})).
They accommodate unbounded likelihood ratios and generalize all known results.
In all cases, it boils down to controlling the behavior of the integrands on an unfavorable event $\{\frac{p_0}{p}>4\}$.

We then compared the sufficient conditions in the literature with our necessary and sufficient conditions.
It was shown that the following implications hold.%
\[
	\text{(\ref{asm:bounded})}\implies\text{(\ref{asm:local})}\implies
	\begin{array}[b]{@{}c@{}}
	\clap{\text{(\ref{eq:BN})}}\\
	\rotatebox[origin=c]{90}{\vphantom{$\centernot\Longleftrightarrow$}$\Longleftrightarrow$}\\
	\text{(\ref{asm:BN})}
	\end{array}
	\iff
	\tikzmarknode{D2}{\text{(\ref{asm:D})}}
	\implies
	\begin{array}[b]{@{}c@{}}
		\clap{\text{(\ref{eq:KLV})}}\\
		\rotatebox[origin=c]{90}{\vphantom{$\centernot\Longleftrightarrow$}$\Longleftrightarrow$}\\
		\text{(\ref{asm:KLV})}
	\end{array}
	\implies
	\begin{array}[b]{@{}c@{}}
		\clap{\text{(\ref{eq:KLD})}}\\
		\rotatebox[origin=c]{90}{\vphantom{$\centernot\Longleftrightarrow$}$\Longleftrightarrow$}\\
		\text{(\ref{asm:KLD})}
	\end{array}
\]

\noindent
The minimal assumption that implies the Hellinger dominance on all three discrepancy measures is (\ref{asm:BN}) with some $\delta\in(0,1]$, and the minimal assumption for the Kullback--Leibler divergence and variation is (\ref{asm:KLV}).

Next, we applied our results to relax the bounded likelihood ratio condition in nonparametric sieve maximum likelihood estimation.
\cref{thm:mle} introduces a new regularity condition (\ref{asm:Mn}), which allows for unbounded likelihood ratios without compromising the convergence rate.
The example based on the normal location model demonstrated the potential usefulness of this generalization.
Other possible applications include posterior contraction rates in nonparametric Bayesian inference and minimax convergence rates in nonparametric estimation.

\bibliographystyle{econ}
\bibliography{hellingerbib}
\addcontentsline{toc}{section}{\refname}

\end{document}